\documentclass[12pt]{amsart}

\newtheorem{theorem}{Theorem}[section]

\newtheorem{corollary}[theorem]{Corollary}

\theoremstyle{definition}

\newtheorem{example}[theorem]{Example}

\begin{document}

\title[]{Strengthened Cauchy-Schwarz and H\"older inequalities}

\author{J. M. Aldaz}
\address{Departamento de Matem\'aticas, Facultad de Ciencias,
Universidad  Aut\'onoma de Madrid, 28049 Madrid, Spain.}
\email{jesus.munarriz@uam.es}

\thanks{2000 {\em Mathematical Subject Classification.} 26D99, 65M60}

\thanks{Keywords and phrases: Strengthened Cauchy-Schwarz, Strengthened H\"older inequality}

\thanks{Partially supported by Grant
MTM2006-13000-C03-03 of the D.G.I. of Spain.}






\maketitle

\begin{abstract} We present some identities related to the
Cauchy-Schwarz inequality in complex inner product spaces. 
A new proof of the basic result on the subject of Strengthened Cauchy-Schwarz
inequalities is derived using these identities. Also, an
analogous version of this result is given for 
Strengthened H\"older
inequalities. 
\end{abstract}

\markboth{J. M. Aldaz}{Cauchy-Schwarz inequality}

\section{Introduction}

In \cite{A}, the parallelogram identity
in a {\em real} inner product space, is rewritten in
Cauchy-Schwarz form 
(with the deviation from equality given as a function of the
angular distance  between
vectors)
thereby providing another proof
of the Cauchy-Schwarz inequality in the real case. The
first section of this note complements this result by
presenting related identities for  complex
inner product spaces, and thus
a proof of the Cauchy-Schwarz inequality  in the complex case. 

Of course, using angular distances
is equivalent to using angles. An advantage of the angular
distance is  that it makes sense in arbitrary normed spaces, in
addition to being simpler than the notion of angle. And in some
cases it may also be easier to compute. Angular distances
are used in Section 2
to give a proof of the basic theorem in the subject of
Strengthened Cauchy-Schwarz inequalities (Theorem \ref{scs} below).
We also point out that the result is valid not just for vector
subspaces, but also for cones. Strengthened Cauchy-Schwarz inequalities
are fundamental in the proofs of convergence of iterative, finite element methods in numerical
analysis, cf. for instance \cite{EiVa}. They have also been considered
in the context of wavelets, cf. for example \cite{DR}, \cite{DRRo1}, \cite{DRRo2}.

Finally, Section 3 presents a variant, for cones and in the H\"older case
when $1 < p < \infty$, of the basic theorem on 
Strengthened Cauchy-Schwarz inequalities, cf. Theorem \ref{sh}.

\section{Identities related to the  Cauchy-Schwarz inequality in complex
inner product spaces}

It is noted in \cite{A} that
in a {\em real} inner product space, the parallelogram identity
\begin{equation}\label{par}
\|x + y\|^2 + \|x-y\|^2 = 2 \|x\|^2 + 2\|y\|^2
\end{equation}
provides
 the following stability version of
 the Cauchy-Schwarz inequality, valid for non-zero vectors
 $x$ and $y$:
\begin{equation}\label{bonhilb}
 (x, y)  =
\|x\|\|y\|\left(1 - \frac12
\left\|\frac{x}{\|x\|}-\frac{y}{\|y\|}\right\|^2\right).
\end{equation}
Basically, this identity says that the size of $(x,y)$
is determined by the angular distance $\left\|\frac{x}{\|x\|}-\frac{y}{\|y\|}\right\|$ between $x$ and $y$.
In particular, $(x, y)  \le
\|x\|\|y\|$, with equality precisely when the angular distance is zero. 

In this section we present some complex variants of this identity, involving 
$(x,y)$ and also $|(x,y)|$; as a byproduct,
  the Cauchy-Schwarz inequality in the complex case is obtained.  Since different conventions 
appear in the literature, we
point out that in this paper $(x,y)$ is taken to be linear in the first argument and
conjugate linear in the second.

We systematically replace in the proofs nonzero vectors
$x$ and $y$ by unit vectors $u = x/\|x\|$ and $v = y/\|y\|$.

\begin{theorem}\label{l1}
For all nonzero vectors $x$ and $y$ in
a complex inner product space, we have
\begin{equation}\label{re}
\operatorname{Re} (x, y)  =
\|x\|\|y\|\left(1 - \frac12
\left\|\frac{x}{\|x\|}-\frac{y}{\|y\|}\right\|^2\right)
\end{equation}
  and  
\begin{equation}\label{im}
\operatorname{Im} (x, y)  =
\|x\|\|y\|\left(1 - \frac12
\left\|\frac{x}{\|x\|}-\frac{i y}{\|y\|}\right\|^2\right).
\end{equation}
\end{theorem}
\begin{proof} Let $\|u\| =\|v\| = 1$. From (\ref{par}) we obtain
 \begin{equation*}
4 - \|u-v\|^2
 =
 \|u + v\|^2
 =
 2 + (u,v) + (v, u)
  =
 2 + (u,v) + \overline{(u, v)}
 =
 2 + 2 \operatorname{Re} (u,v).
 \end{equation*}
 Thus,
$
\operatorname{Re} (u,v) = 1 - \frac12
\left\|u - v \right\|^2.
$
 The same argument, applied to $\|u + i v\|^2$, yields
$
\operatorname{Im} (u,v) = 1 - \frac12
\left\|u - i v \right\|^2.
$
\end{proof}

Writing $(x,y) = \operatorname{Re} (x,y)+ i \operatorname{Im} (x,y)$
we obtain the following

\begin{corollary}\label{complexCauchy}
For all nonzero vectors $x$ and $y$ in
a complex inner product space, we have
\begin{equation}\label{bonchilb1}
(x, y)
  =
\|x\|\|y\|\left(\left(1 - \frac12
\left\|\frac{x}{\|x\|}-\frac{y}{\|y\|}\right\|^2\right)
+
i \left(1 - \frac12
\left\|\frac{x}{\|x\|}-\frac{i y}{\|y\|}\right\|^2\right)\right).
\end{equation}
Thus,
\begin{equation}\label{bonchilb2}
|(x, y)|
  =
\|x\|\|y\|\sqrt{\left(1 - \frac12
\left\|\frac{x}{\|x\|}-\frac{y}{\|y\|}\right\|^2\right)^2
+
\left(1 - \frac12
\left\|\frac{x}{\|x\|}-\frac{i y}{\|y\|}\right\|^2\right)^2}.
\end{equation}
\end{corollary}

Next we find some shorter expressions for $|(x,y)|$. 
Let $\operatorname{Arg} z$ denote the principal argument of $z\in \mathbb{C}$, $z \ne 0$. That
is, $0\le \operatorname{Arg} z < 2\pi$, and in polar coordinates, $z= e^{i\operatorname{Arg} z} r$.  We choose the principal argument for definiteness;
any other argument will do equally well.

\begin{theorem}   Let $x$ and $y$ 
be nonzero vectors in
a complex inner product space. Then, for every $\alpha\in\mathbb{R}$ we have
\begin{equation} \label{stabCauchy}
\|x\|\|y\|\left(1 - \frac12
\left\|\frac{e^{i\alpha}x}{\|x\|}-\frac{y}{\|y\|}\right\|^2\right)
\le 
|(x,y)|
 =
\|x\|\|y\|\left(1 - \frac12
\left\|\frac{e^{ - i \operatorname{Arg} (x,y)} x}{\|x\|}-\frac{y}{\|y\|}\right\|^2\right).
\end{equation}
\end{theorem}

\begin{proof} By a normalization, it is enough to consider
unit vectors $u$ and $v$.
 Let $\alpha$ be an arbitrary real number, and set 
 $t
 = \operatorname{Arg} (u,v)$, so
 $(u,v) = e^{it} r$ in polar form.
Using (\ref{re}) we obtain
$$ 1 - \frac12
\left\|e^{i\alpha} u - v \right\|^2
= \operatorname{Re} (e^{i\alpha} u,v) 
\le 
|(e^{i\alpha} u,v)|
$$
$$
 =
|(u,v)| = r = (e^{-it} u,v) =
\operatorname{Re} (e^{-it} u,v) = 1 - \frac12
\left\|e^{-it} u - v \right\|^2.
$$
\end{proof}

The preceding result can be regarded as a variational expression for
$|(x,y)|$, since it shows that this quantity can be obtained by
maximizing the left hand side of (\ref{stabCauchy}) over $\alpha$, or,
in other words, by minimizing $\left\|\frac{e^{i\alpha}x}{\|x\|}-\frac{y}{\|y\|}\right\|$ over $\alpha$.

\begin{corollary} Cauchy-Schwarz inequality.  For all vectors $x$ and $y$ in
a complex inner product space, we have  $|(x,y)| \le
\left\|x\right\| \|y\|$, with equality if and only if the vectors
are linearly dependent.
\end{corollary}

\begin{proof} Of course, if one of the vectors $x$, $y$ is zero, the result
is trivial, so suppose otherwise and normalize, writing $u =  x/\|x\|$ and
$v = y/\|y\|$. From (\ref{stabCauchy}) we obtain,
first,
$|(u,v)| \le
1$, second, $e^{-i\operatorname{Arg} (u,v)} u = v $
 if $|(u,v)| =  1$, so equality implies linear dependency, and third,  
 $|(u,v)| =  1$  if $e^{i\alpha} u = v$ for some $\alpha\in \mathbb{R}$, so linear dependency implies equality.
 \end{proof}
 
\section{Strengthened Cauchy-Schwarz inequalities} 

Such inequalities, of the form $|(x,y)| \le
\gamma \left\|x\right\| \|y\|$ for some fixed $\gamma \in [0,1)$, are 
fundamental in the proofs of convergence of iterative, finite element methods in numerical
analysis. The basic result in the subject is the following theorem (see Theorem 2.1 and Remark 2.3 of \cite{EiVa}).

\begin{theorem}\label{scs}   Let $H$ be a Hilbert space, let $F\subset H$ be a closed subspace,
and let $V\subset H$ be a finite dimensional subspace. If $F\cap V = \{0\}$, then there exists a constant 
$\gamma = \gamma (V, F)\in [0,1)$
such that for every $x\in V$ and every $y\in F$, 
$$
|(x,y)| \le
\gamma \left\|x\right\| \|y\|.
$$
\end{theorem}

There are, at least, two natural notions of angle between subspaces. To
see this, consider a pair of 2 dimensional
subspaces $V$ and $W$ in $\mathbb{R}^3$, in general position. They intersect
in a line $L$, so we may consider that they are parallel in the
direction of the subspace $L$, and thus the angle between them is
zero. This is the notion of angle relevant to the subject of 
Strengthened Cauchy-Schwarz inequalities.

Alternatively, we may disregard the common subspace $L$, and (in this
particular example) determine the angle between subspaces by choosing the minimal 
angle between their unit normals. Note however that
the two notions of angle suggested by the preceding example
coincide when the intersection of subspaces
is $\{0\}$ (cf. \cite{De} for more information on angles between
subspaces).

From the perspective of angles, or equivalently,  angular distances, what Theorem \ref{scs} states is the intuitively plausible
assertion  that the angular distance between
$V$ and $F$ is strictly positive
provided $F$ is  closed,
 $V$ is finite dimensional, and $F\cap V = \{0\}$.  Finite dimensionality
 of one of the subspaces is crucial, though: It is known that if both
$V$ and $F$ are infinite dimensional, the angular distance between
them can be zero, even if both subspaces are closed. 
 
Define the angular distance between $V$ and $F$ as
\begin{equation} \label{angdist}
\kappa(V, F) := \inf\{\|v - w\|: v\in V, w\in F, \mbox{ and  } \|v\| = \|w\|=1\}.
\end{equation}

The proof (by contradiction) of Theorem  \ref{scs} presented
in \cite{EiVa} is not difficult, but deals only with the case where
both $V$ and $F$ are finite dimensional. And it is certainly not as 
simple as 
the following

\begin{proof} If either $V=\{0\}$ or $F=\{0\}$ there is
nothing to show, so assume otherwise.
Let $S(V)$ be the unit sphere of the finite dimensional subspace $V$, and let $v\in S$. Denote
by $f(v)$ the distance from $v$ to the unit sphere $S(F)$ of $F$. Then
$f(v) > 0$ since $F$ is closed and $v\notin F$. Thus, $f$ achieves
a minimum value $\kappa > 0$ over the compact set $S(V)$. By the right hand side of formula (\ref{stabCauchy}),  for
every $x\in V\setminus \{0\}$ and every $y\in F\setminus \{0\}$ we have $|(x,y)| \le
(1 - \kappa^2/2) \left\|x\right\| \|y\|$.
\end{proof}

In concrete applications of the Strengthened Cauchy-Schwarz inequality, a good deal of effort goes into estimating the size of $\gamma = \cos \theta$, where
$\theta$ is the angle between  subspaces appearing in the discretization
schemes. Since we also have
$\gamma = 1 - \kappa^2/2$, this equality can provide an alternative way of
estimating $\gamma$, via the angular distance $\kappa$ rather than the angle.

Next we state a natural extension of Theorem \ref{scs}, to which the
same proof applies (so we will not repeat it). Consider two nonzero vectors
$u$, $v$ in a real inner product space $E$, and let $S$ the unit circumference in the plane spanned by these vectors. The angle between them is just the
length of the smallest arc of $S$ determined by $u/\|u\|$ and $v/\|v\|$. So to speak about angles, or angular
distances, we only need to be able to multiply nonzero vectors $x$ by  positive
scalars $\lambda = 1/\|x\|$. This suggests that the natural setting
for Theorem \ref{scs} is that of cones, rather than vector subspaces.
Recall that $C$ is a {\em cone} in a vector space over a field containing
the real numbers  if for every $x\in C$ and every
$\lambda > 0$ we have $\lambda x\in C$. In particular, every vector
subspace is a cone. If $C_1$ and $C_2$ are  cones in a 
Hilbert space,  the angular distance between them can be defined
exactly as before:
\begin{equation} \label{angdistc}
\kappa(C_1, C_2) := \inf\{\|v - w\|: v\in C_1, w\in C_2, \mbox{ and  } \|v\| = \|w\|=1\}.
\end{equation}

\begin{theorem}\label{scsc}   Let $H$ be a Hilbert space with unit sphere
$S(H)$, and let $C_1, C_2\subset H$ be (topologically) closed cones, such that $C_1\cap S(H)$
is a norm compact set.  If $C_1\cap C_2 = \{0\}$, then there exists a constant 
$\gamma = \gamma (C_1, C_2)\in [0,1)$
such that for every $x\in C_1$ and every $y\in C_2$, 
$$
|(x,y)| \le
\gamma \left\|x\right\| \|y\|.
$$
\end{theorem}

\begin{example} Let $H=\mathbb{R}^2$, 
$C_1 =\{(x,y) \in \mathbb{R}^2: x= -y\}$ and $C_2 =\{(x,y) \in \mathbb{R}^2: x y\ge 0\}$, that is, $C_1$ is the one dimensional subspace with slope $-1$
and $C_2$ the union of the first and third quadrants. Here we can explicitly
see that $\gamma (C_1, C_2) = \cos (\pi/4) = 1/\sqrt2$. 
However, if $C_2$ is extended to
a vector space $V$, then the condition $C_1\cap V = \{0\}$ no longer holds
and  $\gamma (C_1, V)  = 1$. So stating the result in terms of cones rather
than vector subspaces does cover new, nontrivial cases.
\end{example}

\section{A Strengthened H\"older inequality} 

For $1 < p < \infty$, it is possible to give an $L^p-L^q$ version of
the Strengthened Cauchy-Schwarz inequality. Here $q:= p/(p - 1)$ denotes the
conjugate exponent of $p$. We want to find suitable conditions on 
$C_1\subset L^p$ and $C_2\subset L^q$ so that  
there exists a constant 
$\gamma = \gamma (C_1, C_2)\in [0,1)$
with 
$
\|fg\|_1 \le
\gamma \left\|f\right\|_p \|g\|_q
$
for every $f\in C_1$ and every $g\in C_2$. An obvious difference between
the H\"older and the Cauchy-Schwarz cases is that in the
pairing $(f, g): = \int f\overline{g}$, the functions $f$ and $g$ belong to different
spaces (unless $p = q =2$). This means that the hypothesis $C_1\cap C_2 = \{0\}$ needs to be modified. A second obvious difference is that
H\"older's inequality deals actually with $|f|$ and $|g|$ rather
than with $f$ and $g$. So when finding angular distances we will 
also deal with $|f|$ and $|g|$. Note that $f\in C_i$ does not necessarily
imply that $|f|\in C_i$ (consider, for instance, the second quadrant in $\mathbb{R}^2$). 

We make standard nontriviality assumptions
on measure spaces $(X, \mathcal{A}, \mu)$: $X$ contains at least one point
and the (positive) measure $\mu$ is not identically zero. We write
$L^p$ rather than $L^p (X, \mathcal{A}, \mu)$.

To compare cones in different $L^p$ spaces, we map them into $L^2$ via
the Mazur map. Let us write $\operatorname{ sign } z = e^{i\theta}$ when 
$z = r e^{i\theta} \ne 0$, and $\operatorname{ sign } 0 = 1$ (so 
$|\operatorname{ sign } z| = 1$ always).
The Mazur map  
 $\psi_{r,s} : L^r\to L^s$ is defined first on the
unit sphere $S (L^r)$ by $\psi_{r,s} (f) := |f|^{r/s} \operatorname{ sign }f$, and then extended
to the rest of $L^r$ by homogeneity (cf. \cite{BeLi}, pp. 197--199 for additional 
information on the Mazur map). More precisely, 
$$
\psi_{r,s} (f) := \|f\|_r \psi_{r,s} (f/\|f\|_r) = \|f\|_r^{1 - r/s} |f|^{r/s} \operatorname{ sign }f.
$$
By definition, if $\lambda > 0$ then $\psi_{r,s} (\lambda f) =
\lambda \psi_{r,s} (f)$. This entails that if $C\subset L^r$ is a
cone, then $\psi_{r,s} (C)\subset L^s$ is a cone. Given a subset
$A\subset L^r$, we denote by $|A|$ the set $|A| :=\{|f|: f\in A\}$.
Observe that if $A$ is a cone then so is $|A|$. 

\begin{theorem}\label{sh} Let $1 < p < \infty$ and denote by $q:= p/(p - 1)$  its
conjugate exponent. Let
$C_1\subset L^p$ and $C_2\subset L^q$ be cones, let $S(L^p)$
stand for the unit sphere of $L^p$
and let $\overline{|C_1|}$  and $\overline{|C_2|}$ denote the topological
closures of $|C_1|$  and $|C_2|$. If 
 $\overline{|C_1|}\cap S(L^p)$ is norm compact, and
 $\psi_{p,2} (\overline{|C_1|}) \cap \psi_{q,2} (\overline{|C_2|}) = \{0\}$, then 
  there exists a constant 
$\gamma = \gamma (C_1, C_2)\in [0,1)$
such that for every $f\in C_1$ and every $g\in C_2$, 
\begin{equation}\label{sth}
\|fg\|_1 \le
\gamma \left\|f\right\|_p \|g\|_q.
\end{equation}
\end{theorem}

We use  the part of \cite[Theorem 2.2]{A} given next.

\begin{theorem}\label{betterhold}  Let $1 < p < \infty$, let $q = p/(p-1)$ be its
conjugate exponent, and let $M = \max\{p, q\}$. 
If $f\in L^p $, $g\in L^q$, and $\|f\|_p, \|g\|_q > 0$, then
\begin{equation}\label{bonhold}
\|fg\|_1 \le  \|f\|_p\|g\|_q \left(1 - \frac1M
\left\|\frac{|f|^{p/2}}{\|f\|_p^{p/2}}-\frac{|g|^{q/2}}{\|g\|_q^{q/2}}\right\|_2^2\right).
\end{equation}
\end{theorem}

A different proof of inequality (\ref{bonhold})  (with the slightly weaker
constant $M = p + q$, but sufficient for the purposes of this note) can be found in \cite{A2}. Next we prove Theorem \ref{sh}.

\begin{proof} If either $C_1=\{0\}$ or $C_2=\{0\}$ there is
nothing to show, so assume otherwise. Note that since $|C_1|$ and
$|C_2|$ are cones, the same happens with their topological closures.
The cones $\psi_{p,2} (\overline{|C_1|})$ and $\psi_{q,2} (\overline{|C_2|})$
are also closed, as the following argument shows:
The Mazur maps $\psi_{r,s}$ are uniform homeomorphisms
between closed balls, and also between spheres, of any fixed (bounded) radius  (cf. 
\cite[Proposition 9.2, p. 198]{BeLi}, and the paragraph before the
said proposition). In particular, if $\{f_n\}$ is a Cauchy sequence
in $\psi_{q,2} (\overline{|C_2|})$ (for instance) then it is a bounded 
sequence in $L^2$,
so $\psi_{q,2}^{-1} = \psi_{2,q}$ maps it to a Cauchy sequence in
$\overline{|C_2|}$, with limit, say, $h$. Then $\lim_n f_n = \psi_{q,2} (h) 
\in \psi_{q,2} (\overline{|C_2|})$. Likewise,
$\psi_{p,2} (\overline{|C_1|})$ is closed. 

The rest of the proof
proceeds as before. Let $v\in \psi_{p,2} (\overline{|C_1|})\cap S(L^2)$
and denote
by $F(v)$ the distance from $v$ to  $\psi_{q,2} (\overline{|C_2|})\cap S(L^2)$. Then
$F(v) > 0$, so $F$ achieves
a minimum value $\kappa > 0$ over the compact set $\psi_{p,2} (\overline{|C_1|})\cap S(L^2)$, and now (\ref{sth}) follows
from  (\ref{bonhold}). 
\end{proof}

\end{document}